\documentclass[oneside,reqno,american]{amsart}
\usepackage[T1]{fontenc}
\usepackage[utf8]{inputenc}
\usepackage{xcolor}
\usepackage{babel}
\usepackage{prettyref}
\usepackage{amstext}
\usepackage{amsthm}
\usepackage{amssymb}
\usepackage[pdfusetitle,
 bookmarks=true,bookmarksnumbered=false,bookmarksopen=false,
 breaklinks=false,pdfborder={0 0 1},backref=false,colorlinks=false]
 {hyperref}
\hypersetup{
 colorlinks=true,citecolor=blue,linkcolor=blue,linktocpage=true}

\makeatletter
\numberwithin{equation}{section}
\numberwithin{figure}{section}

\usepackage{prettyref}

\newrefformat{cor}{Corollary~\ref{#1}}
\newrefformat{subsec}{Section~\ref{#1}}
\newrefformat{lem}{Lemma~\ref{#1}}
\newrefformat{thm}{Theorem~\ref{#1}}
\newrefformat{sec}{Section~\ref{#1}}
\newrefformat{chap}{Chapter~\ref{#1}}
\newrefformat{prop}{Proposition~\ref{#1}}
\newrefformat{exa}{Example~\ref{#1}}
\newrefformat{tab}{Table~\ref{#1}}
\newrefformat{rem}{Remark~\ref{#1}}
\newrefformat{def}{Definition~\ref{#1}}
\newrefformat{fig}{Figure~\ref{#1}}
\newrefformat{claim}{Claim~\ref{#1}}

\makeatother

\theoremstyle{plain}
\newtheorem{thm}{\protect\theoremname}[section]
\newtheorem{prop}[thm]{\protect\propositionname}
\newtheorem{cor}[thm]{\protect\corollaryname}
\newtheorem{lem}[thm]{\protect\lemmaname}
\theoremstyle{remark}
\newtheorem*{rem*}{\protect\remarkname}
\newtheorem{rem}[thm]{\protect\remarkname}
\newrefformat{cor}{Corollary \ref{#1}}
\newrefformat{lem}{Lemma \ref{#1}}
\newrefformat{prop}{Proposition \ref{#1}}
\newrefformat{rem}{Remark \ref{#1}}
\newrefformat{sec}{Section~\ref{#1}}
\newrefformat{thm}{Theorem \ref{#1}}
\providecommand{\corollaryname}{Corollary}
\providecommand{\lemmaname}{Lemma}
\providecommand{\propositionname}{Proposition}
\providecommand{\remarkname}{Remark}
\providecommand{\theoremname}{Theorem}

\begin{document}
\title{Universal Kernel Models for Iterated Completely Positive Maps}
\begin{abstract}
We study how iterated and composed completely positive maps act on
operator-valued kernels. Each kernel is realized inside a single Hilbert
space where composition corresponds to applying bounded creation operators
to feature vectors. This model yields a direct formula for every iterated
kernel and allows pointwise limits, contractive behavior, and kernel
domination to be read as standard operator facts. The main results
include an explicit limit kernel for unital maps, a Stein-type decomposition,
a Radon-Nikodym representation under subunitality, and an almost-sure
growth law for random compositions. The construction keeps all iterates
in one space, making their comparison and asymptotic analysis transparent.
\end{abstract}

\author{James Tian}
\address{Mathematical Reviews, 535 W. William St, Suite 210, Ann Arbor, MI
48103, USA}
\email{jft@ams.org}
\keywords{positive definite kernels; operator-valued kernels; reproducing kernel
Hilbert spaces; completely positive maps; Kraus operators; model space;
kernel domination; Radon-Nikodym; random iterations; Lyapunov exponent;
subadditive ergodic theorem.}
\subjclass[2000]{Primary: 47L07; Secondary: 46E22, 47A20, 46L53}

\maketitle
\tableofcontents{}

\section{Introduction}

Positive definite (p.d.) kernels and completely positive (CP) maps
are usually studied in separate settings. Kernels describe geometric
relationships through inner products in reproducing kernel Hilbert
spaces, while CP maps describe structure and dynamics in operator
algebras and quantum channels. This paper connects these ideas by
showing how iterated CP maps act naturally on operator-valued kernels
and how all such iterates can be represented within a single Hilbert
space built from the initial kernel.

A p.d. kernel assigns to each pair of points an operator or scalar
in a way that guarantees positivity of all finite Gram matrices. Such
kernels generate Hilbert spaces where functions are evaluated by inner
products, a construction now standard in functional analysis and machine
learning. (See, e.g,. \cite{MR51437,MR2239907,MR2603770,10.1561/2200000036}.)
Operator-valued kernels extend this picture by allowing values in
$B(H)$, which has been developed in the last two decades in work
on vector-valued learning, operator spaces, and noncommutative function
theory \cite{MR51437,MR2389619,MR3535321}. On the other hand, completely
positive maps are central objects in operator theory and quantum information.
They preserve positivity at all matrix levels, and Stinespring's theorem
describes them by isometries or Kraus operators acting on larger Hilbert
spaces \cite{MR69403,MR247483,MR376726,MR482240}. When such maps
are iterated, they produce linear but generally non-unitary dynamics
on operator algebras and are used to model repeated quantum evolutions
and open-system channels \cite{MR3098923,MR3094122,MR3390682,MR4830256}.

The main motivation of the present work is to understand what happens
when a family of completely positive maps acts repeatedly on an operator-valued
p.d. kernel. Each map modifies the kernel by pushing its operator
entries forward through its Kraus operators. Composing several such
maps corresponds to following different sequences of these transformations.
Rather than studying every composition in isolation, we construct
a single Hilbert space in which all such iterates can be represented
at once. In this space, each map acts by a bounded linear operator,
formally a “creation” operator, on labelled copies of the original
kernel sections. Any finite composition of maps then corresponds to
multiplying the associated creation operators, and the resulting composed
kernel appears as the inner product of the transformed feature vectors.
Within this unified setting, contraction and invariance become operator-norm
properties, and limits or asymptotic projections arise as standard
geometric facts about these shifts.

This geometric view connects to several strands of current research.
In operator theory, it parallels recent analyses of kernel dilations
and Radon-Nikodym representations for completely positive maps \cite{MR2065240,MR2761319,MR2966040,MR3100435,MR3168002,MR3369946,MR3697039}.
In kernel theory, it complements the growing work on operator-valued
and noncommutative kernels, which appear in harmonic analysis, control,
and data-driven systems \cite{MR2394102,MR3100435,MR3535321}. In
quantum information, the iteration of CP maps corresponds to repeated
noisy evolution or to discrete semigroups of quantum channels. Descriptions
of asymptotic fixed points, invariant subspaces, and convergence rates
play a major role there. Our construction expresses these asymptotic
properties directly in kernel form, using standard Hilbert space operations
instead of algebraic manipulations.

The model also provides a clear link with ergodic results for random
compositions of channels. When each step applies a random CP map,
the product of the associated creation operators forms a subadditive
process. Kingman's subadditive ergodic theorem (\cite{MR254907,MR356192,MR3094122,MR797411})
then gives a deterministic Lyapunov exponent controlling the exponential
rate of growth or decay of the iterated kernels. This matches ideas
already present in random matrix products and random quantum channels.
By formulating the problem at the kernel level, we make the comparison
between deterministic and random iteration more transparent.

The paper proceeds by first constructing the universal model and proving
that each iterated kernel has the stated realization (\prettyref{thm:B1}).
We then describe the convergence and Stein-type decomposition for
unital maps (\prettyref{thm:3-1} and \prettyref{thm:3-4}), the Radon-Nikodym
representation for subunital cases (\prettyref{thm:4-2}, with \prettyref{lem:D1}
as the kernel RN tool), and the random iteration results (\prettyref{thm:5-1}
together with Corollaries \ref{cor:5-2}--\ref{cor:5-5}). Throughout,
the emphasis is on keeping all iterates inside one fixed space so
that asymptotic behaviour and comparison become standard geometric
facts (see also Corollaries \ref{cor:3-2}-\ref{cor:3-3} for contraction
bounds).

\section{Model Space and Realization}\label{sec:2}

Our goal here is to build a single Hilbert space where every “composed
kernel” lives and can be compared directly. We start from the given
operator-valued kernel and pass to its scalar lift; this lets us view
each data point $(x,a)$ as a concrete feature vector. Next, for each
completely positive map, we add a simple “creation” step that pushes
those feature vectors forward along a labeled edge. Chaining steps
corresponds to composing maps, where different chains record different
“Kraus choices” along the way. The resulting model space is the orthogonal
sum of all such chains, so inner products separate cleanly by where
the chains land.

The payoff is a transparent realization, where every iterated kernel
is the inner product of pushed-forward feature vectors in this space.
Two immediate benefits follow. First, norm control of the CP maps
translates into operator-norm bounds for the created shifts, hence
for all iterates of the kernel. Second, we obtain a single, universal
stage where limits (deterministic or random), Radon-Nikodym compressions,
and maximality statements can be read off as standard facts about
orthogonal projections and contractions. The formal statement and
proof appear in \prettyref{thm:B1} below. The rest of the paper keeps
returning to this model whenever convergence or comparison is needed.\\

Let $X$ be a set and $H$ a complex Hilbert space. Throughout, all
inner products are linear in the second variable. Let $K:X\times X\to B(H)$
be a positive definite (p.d.) kernel, i.e., for any finite family
$\{(x_{i},h_{i})\}^{m}_{i=1}\subset X\times H$,
\[
\sum^{m}_{i,j=1}\left\langle h_{i},K\left(x_{i},x_{j}\right)h_{j}\right\rangle _{H}\ge0.
\]

Define the scalar kernel $\tilde{K}$ on $X\times H$ by 
\[
\tilde{K}\left(\left(x,a\right),\left(y,b\right)\right):=\left\langle a,K\left(x,y\right)b\right\rangle _{H}.
\]
Indeed, $\tilde{K}:\left(X\times H\right)^{2}\rightarrow\mathbb{C}$
is p.d., since 
\begin{align*}
\sum_{i,j}\overline{c_{i}}c_{j}\tilde{K}\left(\left(x_{i},a_{i}\right),\left(x_{j},a_{j}\right)\right) & =\sum_{i,j}\left\langle c_{i}a_{i},K\left(x_{i},x_{j}\right)c_{j}a_{j}\right\rangle _{H}\\
 & =\sum_{i,j}\left\langle h_{i},K\left(x_{i},x_{j}\right)h_{j}\right\rangle _{H},\quad h_{j}:=c_{j}a_{j},
\end{align*}
which holds for any finite family $\left\{ \left(x_{i},a_{i}\right)\right\} $
in $X\times H$, and $\left\{ c_{i}\right\} $ in $\mathbb{C}$. The
corresponding kernel sections are 
\[
\tilde{K}_{\left(x,a\right)}\left(\cdot\right)=\tilde{K}\left(\cdot,\left(x,a\right)\right)
\]
with 
\begin{equation}
\left\langle \tilde{K}_{\left(x,a\right)},\tilde{K}_{\left(y,b\right)}\right\rangle _{\mathcal{H}_{\tilde{K}}}=\left\langle a,K\left(x,y\right)b\right\rangle _{H}.\label{eq:b1}
\end{equation}
Let $\mathcal{H}_{\tilde{K}}$ be its reproducing kernel Hilbert space
(RKHS), i.e., 
\[
H_{\tilde{K}}=\overline{span}\left\{ \tilde{K}_{\left(x,a\right)}:x\in X,a\in H\right\} .
\]
This is a space of functions $F:X\times H\rightarrow\mathbb{C}$,
with the reproducing property 
\[
F\left(x,a\right)=\langle\tilde{K}_{\left(x,a\right)},F\rangle_{H_{\tilde{K}}},\qquad\left(x,a\right)\in X\times H.
\]

Moreover, for all $x\in X$, setting $V_{x}:H\rightarrow H_{\tilde{K}}$
by $V_{x}a=\tilde{K}_{\left(x,a\right)}$ through the kernel sections,
then \eqref{eq:b1} gives the Kolmogorov decomposition 
\[
K\left(x,y\right)=V^{*}_{x}V_{y},\qquad x,y\in X.
\]
For details of this scalar `trick', we refer to \cite{MR2938971}
and the recent survey \cite{MR4250453}\@.

Fix a nonempty index set $S$. Let $S^{*}$ denote the free monoid
on $S$ (finite words, including the empty word $\emptyset$), and
use juxtaposition $ww'$ for concatenation. 

For each $s\in S$, let $\Phi_{s}:B\left(H\right)\rightarrow B\left(H\right)$
be a normal completely positive (CP) map, with a (possibly countable)
Kraus family 
\[
\Phi_{s}\left(T\right)=\sum_{r\in I\left(s\right)}A^{*}_{s,r}TA_{s,r}
\]
(converging in the strong operator topology) If $\Phi_{s}$ is unital,
then $\sum_{r\in I\left(s\right)}A^{*}_{s,r}A_{s,r}=I$ (strongly). 

For $w=s_{1}s_{2}\cdots s_{n}\in S^{*}$, define the iterated kernel
by
\[
K_{w}:=\Phi_{s_{1}}\circ\Phi_{s_{2}}\circ\cdots\circ\Phi_{s_{n}}\left(K\right),\qquad K_{\emptyset}:=K.
\]

Define the index set 
\[
R:=\bigsqcup_{w=s_{1}\cdots s_{n}\in S^{*}}I\left(s_{1}\right)\times\cdots\times I\left(s_{n}\right),
\]
with the convention that for the empty word $\emptyset$, the corresponding
factor is the singleton $\left\{ \emptyset\right\} $. Thus, elements
of $R$ are finite index strings $\rho=\left(r_{1},\dots,r_{n}\right)$
with $r_{j}\in I\left(s_{j}\right)$ for some $w=s_{1}\cdots s_{n}\in S^{*}$,
and include the empty string $\emptyset$. For $r\in I\left(s\right)$
and an index string $\rho$, write $r\rho$ for left-concatenation
(so $r\emptyset=r$). 

Define the model space 
\[
M:=l^{2}\left(R\right)\otimes\mathcal{H}_{\tilde{K}}\simeq\bigoplus_{\rho\in R}\mathcal{H}_{\tilde{K}}.
\]
Let $\left\{ e_{\rho}\right\} _{\rho\in R}$ denote the canonical
orthonormal basis (ONB) for $l^{2}\left(R\right)$. We will use the
notation $e_{\rho}\cdot F$ to denote the elementary tensor $e_{\rho}\otimes F$,
which corresponds to the element $F$ in the $\rho$-th summand. 

Define $J_{x}:H\rightarrow M$ by 
\[
J_{x}\left(a\right):=e_{\emptyset}\cdot\tilde{K}_{\left(x,a\right)}.
\]

For each $s\in S$, define a bounded operator $C_{s}:M\rightarrow M$
on the algebraic span by 
\[
C_{s}\left(e_{\rho}\cdot\tilde{K}_{\left(x,a\right)}\right)=\sum_{r\in I\left(s\right)}e_{r\rho}\cdot\tilde{K}_{\left(x,A_{s,r}a\right)}.
\]

For any $w=s_{1}\cdots s_{n}\in S^{*}$, set the right-regular product
\begin{equation}
C_{w}:=C_{s_{n}}\cdots C_{s_{1}},\qquad C_{\emptyset}:=I_{M}.\label{eq:B1-1}
\end{equation}

\begin{thm}
\label{thm:B1}Let $K$, $\{\Phi_{s}\}$, and $\{K_{w}\}$ be as defined
above. Then the following hold:
\begin{enumerate}
\item \label{enu:B1-1}$K_{w}$ is a $B\left(H\right)$-valued p.d. kernel. 
\item \label{enu:B1-2}Each $C_{s}$ is bounded with $\|C_{s}\|\le\|\Phi_{s}\|^{1/2}_{\mathrm{cb}}$,
and $\|C_{s}\|=1$ if $\Phi_{s}$ is unital. Consequently, 
\begin{equation}
\left\Vert C_{w}\right\Vert \leq\prod^{n}_{j=1}\left\Vert \Phi_{s_{j}}\right\Vert ^{1/2}_{\mathrm{cb}},\quad\forall w=s_{1}\cdots s_{n}\in S^{*}.\label{eq:B1}
\end{equation}
\item \label{enu:B1-3}$K_{w}$ admits the representation 
\begin{equation}
K_{w}\left(x,y\right)=J^{*}_{x}C^{*}_{w}C_{w}J_{y}.\label{eq:B2}
\end{equation}
Equivalently, 
\begin{equation}
\left\langle a,K_{w}\left(x,y\right)b\right\rangle _{H}=\left\langle C_{w}J_{x}\left(a\right),C_{w}J_{y}\left(b\right)\right\rangle _{M}.\label{eq:B3}
\end{equation}
\end{enumerate}
\end{thm}

\begin{proof}
\eqref{enu:B1-1} Fix $w=s_{1}\cdots s_{n}\in S^{*}$ and a finite
family $\left\{ \left(x_{i},h_{i}\right)\right\} ^{m}_{i=1}$ in $X\times H$.
The block matrix $\left[K\left(x_{i},x_{j}\right)\right]\in M_{m}\left(B\left(H\right)\right)$
is positive. Each $\Phi_{s_{k}}$ is completely positive, so $\left(id_{M_{m}}\otimes\Phi_{s_{k}}\right)$
preserves positivity. Iterating for $k=n,\dots,1$ yields $\left[K_{w}\left(x_{i},x_{j}\right)\right]\geq0$. 

For \eqref{enu:B1-2}, take a finite vector $u=\sum_{i}e_{\rho_{i}}f_{i}\in M$,
with 
\[
f_{i}=\sum_{j}c_{ij}\tilde{K}_{\left(x_{ij},a_{ij}\right)}\in\mathcal{H}_{\tilde{K}}.
\]
Using orthogonality of the direct-sum summands and the reproducing
property, 
\begin{align*}
\left\Vert C_{s}u\right\Vert ^{2}_{M} & =\sum_{i}\sum_{r\in I\left(s\right)}\Big\Vert\sum_{j}c_{ij}\tilde{K}_{\left(x_{ij},A_{s,r}a_{ij}\right)}\Big\Vert^{2}_{\mathcal{H}_{\tilde{K}}}\\
 & =\sum_{i}\sum_{r\in I\left(s\right)}\sum_{j,k}\overline{c_{ij}}c_{ik}\left\langle \tilde{K}_{\left(x_{ij},A_{s,r}a_{ij}\right)},\tilde{K}_{\left(x_{ik},A_{s,r}a_{ik}\right)}\right\rangle _{\mathcal{H}_{\tilde{K}}}\\
 & =\sum_{i}\sum_{r\in I\left(s\right)}\sum_{j,k}\overline{c_{ij}}c_{ik}\left\langle A_{s,r}a_{ij},K\left(x_{ij},x_{ik}\right)A_{s,r}a_{ik}\right\rangle _{H}\\
 & =\sum_{i}\sum_{j,k}\overline{c_{ij}}c_{ik}\left\langle a_{ij},\Phi_{s}\left(K\left(x_{ij},x_{ik}\right)\right)a_{ik}\right\rangle _{H}.
\end{align*}
Recall that, for any positive block $\left[T_{jk}\right]\in M_{N}\left(B\left(H\right)\right)$
and $\alpha\in H^{N}$, one has
\[
\left\langle \alpha,\left[\Phi_{s}\left(T_{jk}\right)\right]\alpha\right\rangle \leq\left\Vert \Phi_{s}\right\Vert _{\mathrm{cb}}\left\langle \alpha,\left[T_{jk}\right]\alpha\right\rangle .
\]
Applying this to $T_{jk}=K\left(x_{ij},x_{ik}\right)$ gives 
\[
\left\Vert C_{s}u\right\Vert ^{2}_{M}\leq\left\Vert \Phi_{s}\right\Vert _{\mathrm{cb}}\left\Vert u\right\Vert ^{2}_{M},
\]
hence $\left\Vert C_{s}\right\Vert _{M\rightarrow M}\leq\left\Vert \Phi\right\Vert ^{1/2}_{\mathrm{cb}}$.
If $\Phi_{s}$ is unital, then $\left\Vert \Phi_{s}\right\Vert _{\mathrm{cb}}=1$
and $\left\Vert C_{s}\right\Vert _{M\rightarrow M}=1$. The bound
for $C_{w}$ as in \eqref{eq:B1} then follows from submultiplicativity. 

For \eqref{enu:B1-3}, let $w=s_{1}\cdots s_{n}\in S^{*}$. From $J_{x}\left(a\right)=e_{\emptyset}\cdot\tilde{K}_{\left(x,a\right)}$
and the right-regular product $C_{w}=C_{s_{n}}\cdots C_{s_{1}}$,
we have 
\[
C_{w}J_{x}\left(a\right)=\sum_{\left(r_{1},\dots,r_{n}\right)}e_{\left(r_{n},\dots,r_{1}\right)}\cdot\tilde{K}_{\left(x,A_{s_{n},r_{n}}\cdots A_{s_{1},r_{1}}a\right).}
\]
By orthogonality of the direct-sum summands and the RKHS inner product,
\begin{align*}
\left\langle C_{w}J_{x}\left(a\right),C_{w}J_{y}\left(b\right)\right\rangle _{M} & =\sum_{\left(r_{1},\dots,r_{n}\right)}\left\langle \tilde{K}_{\left(x,A_{s_{n},r_{n}}\cdots A_{s_{1},r_{1}}a\right)},\tilde{K}_{\left(y,A_{s_{n},r_{n}}\cdots A_{s_{1},r_{1}}b\right)}\right\rangle _{\mathcal{H}_{\tilde{K}}}\\
 & =\sum_{\left(r_{1},\dots,r_{n}\right)}\left\langle A_{s_{n},r_{n}}\cdots A_{s_{1},r_{1}}a,K\left(x,y\right)A_{s_{n},r_{n}}\cdots A_{s_{1},r_{1}}b\right\rangle _{H}\\
 & =\sum_{\left(r_{1},\dots,r_{n}\right)}\left\langle a,A^{*}_{s_{1},r_{1}}\cdots A^{*}_{s_{n},r_{n}}K\left(x,y\right)A_{s_{n},r_{n}}\cdots A_{s_{1},r_{1}}b\right\rangle _{H}\\
 & =\left\langle a,\left(\Phi_{s_{1}}\circ\cdots\circ\Phi_{s_{n}}\right)\left(K\left(x,y\right)\right)b\right\rangle _{H}\\
 & =\left\langle a,K_{w}\left(x,y\right)b\right\rangle _{H}.
\end{align*}
The stated identities \eqref{eq:B2}-\eqref{eq:B3} follow.
\end{proof}
Let $\mathcal{D}$ denote the algebraic span of the vectors $\left\{ C_{w}J_{x}a:w\in S^{*},x\in X,a\in H\right\} $
inside the ambient model $M$. Equip $\mathcal{D}$ with the sesquilinear
form 
\begin{equation}
\left\langle C_{w}J_{x}a,C_{v}J_{y}b\right\rangle _{\mathcal{D}}:=\left\langle C_{w}J_{x}a,C_{v}J_{y}b\right\rangle _{M}.\label{eq:B-6}
\end{equation}
A triple $(\tilde{M},\{\tilde{C}_{s}\},\{\tilde{J}_{x}\})$ is called
an admissible realization if the linear map 
\[
\mathcal{D}\longrightarrow\tilde{M},\qquad C_{w}J_{x}a\longmapsto\tilde{C}_{w}\tilde{J}_{x}a
\]
is isometric with respect to the form \eqref{eq:B-6}. It is minimal
if 
\[
\tilde{M}=\overline{span}\left\{ \tilde{C}_{w}\tilde{J}_{x}a:w\in S^{*},x\in X,a\in H\right\} .
\]

\begin{prop}
Let $\left(M,\{C_{s}\},\{J_{x}\}\right)$ be the model constructed
in \prettyref{thm:B1}, and set 
\[
M_{0}:=\overline{span}\left\{ C_{w}J_{x}a:w\in S^{*},x\in X,a\in H\right\} \subset M.
\]
Then $M_{0}$ is invariant under each $C_{s}$ and provides a minimal
admissible realization. If $(\tilde{M},\{\tilde{C}_{s}\},\{\tilde{J}_{x}\})$
is another minimal admissible realization, there exists a unique unitary
operator 
\[
U:M_{0}\longrightarrow\tilde{M}
\]
such that $UJ_{x}=\tilde{J}_{x}$ and $UC_{s}=\tilde{C}_{s}U$ for
all $x\in X$, $s\in S$. 
\end{prop}

\begin{proof}
Since $C_{s}\left(C_{w}J_{x}a\right)=C_{ws}J_{x}a\in\mathcal{D}$
and $J_{x}a=C_{\emptyset}J_{x}a\in\mathcal{D}$, we have $M_{0}=\overline{D}$
and $M_{0}$ is $C_{s}$-invariant. 

Define $U_{0}:\mathcal{D}\rightarrow\tilde{M}$ by $U_{0}\left(C_{w}J_{x}a\right):=\tilde{C}_{w}\tilde{J}_{x}a$
and extend linearly. For $u=\sum_{i}\alpha_{i}C_{w_{i}}J_{x_{i}}a_{i}$
and $v=\sum_{j}\beta_{j}C_{v_{j}}J_{y_{j}}b_{j}$ in $\mathcal{D}$,
\begin{align*}
\left\langle u,v\right\rangle _{M} & =\sum_{i,j}\overline{\alpha_{i}}\beta_{j}\left\langle C_{w_{i}}J_{x_{i}}a_{i},C_{v_{j}}J_{y_{j}}b_{j}\right\rangle _{M}\\
 & =\sum_{i,j}\overline{\alpha_{i}}\beta_{j}\left\langle \tilde{C}_{w_{i}}\tilde{J}_{x_{i}}a_{i},\tilde{C}_{v_{j}}\tilde{J}_{y_{j}}b_{j}\right\rangle _{M}\qquad\left(\text{admissibility}\right)\\
 & =\left\langle U_{0}u,U_{0}v\right\rangle _{\tilde{M}}.
\end{align*}
Since $\mathcal{D}$ is dense in $M_{0}$, $U_{0}$ extends uniquely
by continuity to an isometry $U:M_{0}\rightarrow\overline{U_{0}\left(\mathcal{D}\right)}=\tilde{M}$,
by minimality of $\tilde{M}$. 

For each $x\in X$ and $a\in H$, 
\[
UJ_{x}a=U\left(C_{\emptyset}J_{x}a\right)=\tilde{C}_{\emptyset}\tilde{J}_{x}a=\tilde{J}_{x}a,
\]
and for $a\in S$, 
\[
UC_{s}\left(C_{w}J_{x}a\right)=U\left(C_{ws}J_{x}a\right)=\tilde{C}_{ws}\tilde{J}_{x}a=\tilde{C}_{s}\tilde{C}_{w}\tilde{J}_{x}a=\tilde{C}_{s}U\left(C_{w}J_{x}a\right).
\]
By linearity and continuity these relations hold on all of $M_{0}$.
Uniqueness of $U$ follows from the density of $\mathcal{D}$. 
\end{proof}

\section{Iteration and Limit Kernel}\label{sec:3}

This section deals with what repeated application of the same completely
positive map does to the kernel. In the model from \prettyref{sec:2},
the process is simply pushing feature vectors forward by a single
contraction, and the geometry forces a clear outcome, where the non-decaying
component survives while the transient part fades. Passing back to
kernels, the iterates converge pointwise to a limit obtained by projecting
feature vectors onto the non-decaying subspace and taking their inner
products there. This yields an explicit limit formula and a canonical
Stein decomposition of the original kernel into its invariant part
and a positive defect series.\\

Fix a single normal completely positive map $\Phi:B\left(H\right)\rightarrow B\left(H\right)$
with Kraus family $\Phi\left(T\right)=\sum_{r}A^{*}_{r}TA_{r}$ (strongly
convergent). Set $K_{n}:=\Phi^{n}\left(K\right)$ and define $C:=C_{\Phi}$
by 
\[
C\left(e_{\rho}\cdot\tilde{K}_{\left(x,a\right)}\right)=\sum_{r}e_{r\rho}\cdot\tilde{K}_{\left(x,A_{r}a\right)}.
\]
By \prettyref{thm:B1}, $\left\Vert C\right\Vert _{M\rightarrow M}\leq$$\left\Vert \Phi\right\Vert ^{1/2}_{\mathrm{cb}}$
and 
\[
\left\langle a,K_{n}\left(x,y\right)b\right\rangle _{H}=\left\langle C^{n}J_{x}\left(a\right),C^{n}J_{y}\left(b\right)\right\rangle _{M}.
\]
 
\begin{thm}
\label{thm:3-1}Assume $\Phi:B\left(H\right)\rightarrow B\left(H\right)$
is unital completely positive. Then for all $x,y\in X$ and $a,b\in H$,
\begin{equation}
\lim_{n\rightarrow\infty}\left\langle a,K_{n}\left(x,y\right)b\right\rangle _{H}=\left\langle a,\overline{K}\left(x,y\right)b\right\rangle _{H},
\end{equation}
where 
\begin{equation}
\overline{K}\left(x,y\right)=J^{*}_{x}PJ_{y},\label{eq:c2}
\end{equation}
and
\begin{equation}
P=s\text{-}\lim_{n\rightarrow\infty}C^{*n}C^{n},\label{eq:c3}
\end{equation}
which is the orthogonal projection onto the isometric subspace 
\begin{equation}
M_{iso}:=\left\{ v\in M:\left\Vert C^{n}v\right\Vert =\left\Vert v\right\Vert ,\;\forall n\right\} .\label{eq:C4}
\end{equation}
In particular, $\overline{K}$ is a $B\left(H\right)$-valued p.d.
kernel.
\end{thm}

\begin{proof}
\textbf{Step 1.} By \prettyref{thm:B1}, $K_{n}\left(x,y\right)=J^{*}_{x}C^{*n}C^{n}J_{y}$.
Unital CP implies $\left\Vert C\right\Vert _{M\rightarrow M}\leq1$.
For $n\geq0$, set 
\[
T_{n}:=C^{*n}C^{n}\in B\left(M\right).
\]
Each $T_{n}$ is a positive contraction, and 
\[
T_{n+1}=C^{*n}\left(C^{*}C\right)C^{n}\leq C^{*n}IC^{n}=T_{n},
\]
so $\left(T_{n}\right)_{n\geq0}$ is monotone decreasing in the operator
order, with $0\leq T_{n}\leq I_{M}$ . 

For any $v\in M$, the scalar sequence $\left\langle v,T_{n}v\right\rangle $
is decreasing and bounded below by $0$, hence convergent. Define
a quadratic form
\[
q\left(v\right):=\lim_{n\rightarrow\infty}\left\langle v,T_{n}v\right\rangle ,\quad v\in M
\]
Since $0\leq T_{n}\leq I_{M}$, then $0\leq q\left(v\right)\leq\left\Vert v\right\Vert ^{2}$
and so $q$ is bounded. By the representation theorem for bounded
positive quadratic forms on a Hilbert space, there exists a unique
bounded positive operator $P$ with $0\leq P\leq I$ such that 
\[
q\left(v\right)=\left\langle v,Pv\right\rangle ,\quad v\in M.
\]
Claim: $T_{n}\xrightarrow{s}P$ on $M$. 

Indeed, set $S_{n}:=T_{n}-P$. Then $S_{n}\geq0$ and $S_{n+1}\leq S_{n}$
(since $T_{n}\downarrow P$), hence $\left\Vert S_{n}\right\Vert \leq\left\Vert S_{0}\right\Vert =\left\Vert I_{M}-P\right\Vert $.
For any $v\in M$, 
\[
\left\Vert S_{n}v\right\Vert ^{2}=\left\langle v,S^{2}_{n}v\right\rangle \leq\left\Vert S_{n}\right\Vert \left\langle v,S_{n}v\right\rangle \leq\left\Vert S_{0}\right\Vert \left\langle v,S_{n}v\right\rangle ,
\]
where we used the operator inequality $A^{2}\leq\left\Vert A\right\Vert A$
for any positive operator $A$. 

By the definition of $P$, we have $\left\langle v,S_{n}v\right\rangle \rightarrow0$
for every $v\in M$, hence $\left\Vert S_{n}v\right\Vert \rightarrow0$,
for all $v\in M$. Therefore $S_{n}\rightarrow0$ in the strong operator
topology, i.e., $T_{n}\rightarrow P$ strongly. This is \eqref{eq:c3}.

\textbf{Step 2.} Let $M_{iso}$ be as in \eqref{eq:C4}. For $v\in M_{iso}$,
we have $\left\langle v,T_{n}v\right\rangle =\left\Vert C^{n}v\right\Vert ^{2}=\left\Vert v\right\Vert ^{2}$,
hence $q\left(v\right)=\left\langle v,Pv\right\rangle =\left\Vert v\right\Vert ^{2}$
and therefore $Pv=v$ (using the fact that $P$ is a positive contraction).
Thus $M_{iso}\subset ran\left(P\right)$. Conversely, suppose $Pv=v$.
Then $\left\Vert C^{n}v\right\Vert ^{2}=\left\langle v,T_{n}v\right\rangle \rightarrow\left\langle v,Pv\right\rangle =\left\Vert v\right\Vert ^{2}$.
But $\left\Vert C^{n}v\right\Vert $ is a decreasing sequence (since
$\left\Vert C\right\Vert \leq1$), so $\left\Vert C^{n}v\right\Vert =\left\Vert v\right\Vert $
for all $n$. Hence $v\in M_{iso}$. Thus $ran\left(P\right)=M_{iso}$,
which shows that $P$ is the orthogonal projection onto $M_{iso}$. 

\textbf{Step 3.} Using the realization $K_{n}\left(x,y\right)=J^{*}_{x}C^{*n}C^{n}J_{y}$,
we have 
\[
\left\langle a,K_{n}\left(x,y\right)b\right\rangle _{H}=\left\langle J_{x}a,T_{n}J_{y}b\right\rangle _{M}\longrightarrow\left\langle J_{x}\left(a\right),PJ_{y}\left(b\right)\right\rangle _{M}=\left\langle a,\overline{K}\left(x,y\right)b\right\rangle _{H}.
\]
This is \eqref{eq:c2}.
\end{proof}
\begin{cor}[strict contraction]
\label{cor:3-2} Let $\Phi:B\left(H\right)\rightarrow B\left(H\right)$
be normal completely positive with $\left\Vert \Phi\right\Vert _{\mathrm{cb}}<1$.
Set $K_{n}:=\Phi^{n}\left(K\right)$ and let $C:=C_{\Phi}$ be as
above. Then:
\begin{enumerate}
\item For all $x,y\in X$ and $a,b\in H$, 
\begin{equation}
\left|\left\langle a,K_{n}\left(x,y\right)b\right\rangle _{H}\right|\leq\left\Vert \Phi\right\Vert ^{n}_{\mathrm{cb}}\left\Vert J_{x}\left(a\right)\right\Vert _{M}\left\Vert J_{y}\left(b\right)\right\Vert _{M}.\label{eq:C5}
\end{equation}
\item $C^{*n}C^{n}\xrightarrow{s}0$ on $M$. 
\item The following norm estimate holds: 
\begin{equation}
\left\Vert K_{n}\left(x,y\right)\right\Vert _{B\left(H\right)}\leq\left\Vert \Phi\right\Vert ^{n}_{\mathrm{cb}}\sqrt{\left\Vert K\left(x,x\right)\right\Vert \left\Vert K\left(y,y\right)\right\Vert }.\label{eq:C6}
\end{equation}
\end{enumerate}
\end{cor}

\begin{proof}
By \eqref{thm:B1}, for all $x,y,a,b$, 
\[
\left\langle a,K_{n}\left(x,y\right)b\right\rangle _{H}=\left\langle C^{n}J_{x}\left(a\right),C^{n}J_{y}\left(b\right)\right\rangle _{M}.
\]
Let $r:=\left\Vert C\right\Vert $. Part \eqref{enu:B1-2} of \eqref{thm:B1}
gives $r\leq\left\Vert \Phi\right\Vert ^{1/2}_{\mathrm{cb}}<1$. 

By Cauchy-Schwarz, 
\begin{align*}
\left|\left\langle a,K_{n}\left(x,y\right)b\right\rangle _{H}\right| & \leq\left\Vert C^{n}J_{x}\left(a\right)\right\Vert _{M}\left\Vert C^{n}J_{y}\left(b\right)\right\Vert _{M}\\
 & \leq\left\Vert C^{n}\right\Vert ^{2}_{M}\left\Vert J_{x}\left(a\right)\right\Vert _{M}\left\Vert J_{y}\left(b\right)\right\Vert _{M},
\end{align*}
where $\left\Vert C^{n}\right\Vert ^{2}_{M}\leq r^{2n}\leq\left\Vert \Phi\right\Vert ^{n}_{\mathrm{cb}}$.
This is \eqref{eq:C5}.

Since $r<1$, $\left\Vert C^{n}v\right\Vert _{M}\leq r^{n}\left\Vert v\right\Vert _{M}\rightarrow0$
for ever $v\in M$. Hence $C^{n}\rightarrow0$ strongly. Then for
any $v\in M$, 
\[
\left\Vert C^{n*}C^{n}v\right\Vert _{M}\leq r^{2n}\left\Vert v\right\Vert \rightarrow0,
\]
so $C^{*n}C^{n}\xrightarrow{s}0$. 

Finally, it follows from \eqref{eq:C5} that 
\begin{equation}
\left\Vert K_{n}\left(x,y\right)\right\Vert \leq\left\Vert \Phi\right\Vert ^{n}_{\mathrm{cb}}\left\Vert J_{x}\right\Vert _{H\rightarrow M}\left\Vert J_{y}\right\Vert _{H\rightarrow M}.\label{eq:C7}
\end{equation}
By the RKHS construction, 
\begin{align*}
\left\Vert J_{x}\right\Vert ^{2}_{H\rightarrow M} & =\sup_{\left\Vert a\right\Vert _{H}=1}\left\Vert J_{x}\left(a\right)\right\Vert ^{2}_{M}\\
 & =\sup_{\left\Vert a\right\Vert _{H}=1}\left\langle a,K\left(x,x\right)a\right\rangle _{H}=\left\Vert K\left(x,x\right)\right\Vert _{B\left(H\right)},
\end{align*}
and similarly $\left\Vert J_{y}\right\Vert ^{2}_{M}=\left\Vert K\left(y,y\right)\right\Vert _{B\left(H\right)}$.
Substituting this into \eqref{eq:C7} gives the stated bound in \eqref{eq:C6}. 
\end{proof}
\begin{cor}
\label{cor:3-3}Fix $S,S^{*}$ as in \prettyref{sec:2}, and let $\left\{ \Phi_{s}\right\} _{s\in S}$
be normal completely positive maps with 
\[
\sup_{s\in S}\left\Vert \Phi_{s}\right\Vert _{\mathrm{cb}}\leq L<1.
\]
For any $w=s_{1}\cdots s_{n}\in S^{*}$, set $K_{w}:=\Phi_{s_{1}}\circ\cdots\circ\Phi_{s_{n}}\left(K\right)$
and $C_{w}:=C_{s_{n}}\cdots C_{s_{1}}$. Then: 
\begin{enumerate}
\item $\left\Vert C_{w}\right\Vert _{M}\leq\prod^{n}_{j=1}\left\Vert \Phi_{s_{j}}\right\Vert ^{1/2}_{\mathrm{cb}}\leq L^{n/2}$. 
\item For all $x,y\in X$ and $a,b\in H$, 
\[
\left|\left\langle a,K_{w}\left(x,y\right)b\right\rangle _{H}\right|\leq L^{n}\left\Vert J_{x}\left(a\right)\right\Vert _{M}\left\Vert J_{y}\left(b\right)\right\Vert _{M}.
\]
\item Moreover, 
\[
\left\Vert K_{w}\left(x,y\right)\right\Vert \leq L^{\left|w\right|}\sqrt{\left\Vert K\left(x,x\right)\right\Vert \left\Vert K\left(y,y\right)\right\Vert }.
\]
\end{enumerate}
\end{cor}

\begin{proof}
Immediate from the above discussion. 
\end{proof}
We now refine the asymptotic picture from \prettyref{thm:3-1} by
separating the initial kernel $K$ into a $\Phi$-harmonic component
and a transient part. Writing $Q:=K-\Phi\left(K\right)$, the kernel
$Q$ measures the one-step “defect’’ of $K$ under $\Phi$. Iterating
this defect yields a canonical series that accounts for all non-harmonic
contributions of $K$ to the orbit $K_{n}=\Phi^{n}(K)$. The next
result shows that $K-\overline{K}$ is the monotone sum of these iterates,
that $\overline{K}$ is $\Phi$-harmonic, and that it is maximal among
$\Phi$-harmonic kernels dominated by $K$.

This discrete-time decomposition parallels the classical Stein (or
discrete Lyapunov) equation and its Neumann-series solution, standard
in operator and matrix theory (see e.g., \cite{MR2284176,MR3532794,MR1367089}).
Related notions of $\Phi$-harmonic elements and noncommutative Poisson
boundaries appear in e.g., \cite{MR482240,MR1916370}.
\begin{thm}[Stein decomposition]
\label{thm:3-4}Let $X$ be a set, $H$ a complex Hilbert space,
$K:X\times X\rightarrow B\left(H\right)$ a p.d. kernel, and $\Phi:B\left(H\right)\rightarrow B\left(H\right)$
a normal unital completely positive map. Set $K_{n}:=\Phi^{n}\left(K\right)$
and let $\overline{K}$ be the limit kernel given by \prettyref{thm:3-1},
so $K_{n}\left(x,y\right)\xrightarrow{w}\overline{K}\left(x,y\right)$
for each $x,y\in X$ (in the weak operator topology ). Define 
\[
Q\left(x,y\right):=K\left(x,y\right)-\Phi\left(K\right)\left(x,y\right).
\]
Then:
\begin{enumerate}
\item For all $x,y\in X$, 
\begin{equation}
K\left(x,y\right)-\overline{K}\left(x,y\right)=\sum^{\infty}_{j=0}\Phi^{j}\left(Q\right)\left(x,y\right),\label{eq:C8}
\end{equation}
where the series converges pointwise in the weak operator topology
and, moreover, in the p.d. order as a monotone increasing sequence
of kernels.
\item The limit kernel $\overline{K}$ is $\Phi$-harmonic, i.e., 
\begin{equation}
\Phi\left(\overline{K}\right)=\overline{K}.\label{eq:C9}
\end{equation}
\item If $L:X\times X\rightarrow B\left(H\right)$ is $\Phi$-harmonic and
$0\preceq L\preceq K$, then $L\preceq\overline{K}$. Equivalently,
$\overline{K}$ is the largest $\Phi$-harmonic kernel dominated by
$K$. 
\end{enumerate}
\end{thm}

\begin{proof}
Set $S_{N}:=\sum^{N-1}_{j=0}\Phi^{j}\left(Q\right)$ for $N>1$. Then
\[
K-\Phi^{N}\left(K\right)=\sum^{N-1}_{j=0}\left(\Phi^{j}\left(K\right)-\Phi^{j+1}\left(K\right)\right)=\sum^{N-1}_{j=0}\Phi^{j}\left(Q\right)=S_{N}.
\]
Each $\Phi^{j}\left(Q\right)$ is p.d., hence $S_{N}$ is p.d. and
$S_{N}\preceq S_{N+1}$. 

By \prettyref{thm:3-1}, $\Phi^{N}\left(K\right)=K_{N}\xrightarrow{w}\overline{K}$
pointwise in the weak operator topology, hence 
\[
S_{N}\left(x,y\right)=K\left(x,y\right)-\Phi^{N}\left(K\right)\left(x,y\right)\xrightarrow{w}K\left(x,y\right)-\overline{K}\left(x,y\right).
\]
Because $\left(S_{N}\right)_{N\geq1}$ is monotone increasing in the
p.d. order and is bounded above by $K$, the limit exists also in
the p.d. order and equals $K-\overline{K}$. This proves \eqref{eq:C8}.

Apply $\Phi$ to the identity \eqref{eq:C8}, so that
\[
\Phi\left(K\right)-\Phi\left(\overline{K}\right)=\sum^{\infty}_{j=1}\Phi^{j}\left(Q\right).
\]
Adding $Q=K-\Phi\left(K\right)$ to both sides gives 
\[
K-\Phi\left(\overline{K}\right)=Q+\sum^{\infty}_{j=1}\Phi^{j}\left(Q\right)=\sum^{\infty}_{j=0}\Phi^{j}\left(Q\right).
\]
This implies that $K-\Phi\left(\overline{K}\right)=K-\overline{K}$,
hence $\Phi\left(\overline{K}\right)=\overline{K}$, which is \eqref{eq:C9}.

Suppose $L$ is $\Phi$-harmonic and $0\preceq L\preceq K$. Then
$K-L$ is p.d., and complete positivity gives $\Phi^{N}\left(K-L\right)\succeq0$
for all $N$. Iterating the identity 
\[
K-L=\Phi\left(K-L\right)+Q
\]
gives 
\[
K-L=\Phi^{N}\left(K-L\right)+S_{N}.
\]
Since $\Phi^{N}\left(K-L\right)\succeq0$ we obtain the monotone lower
bound
\[
K-L\succeq S_{N},\quad\forall N.
\]
Taking $N\rightarrow\infty$ and using $S_{N}\xrightarrow{w}K-\overline{K}$
gives 
\[
K-L\succeq K-\overline{K},
\]
which is equivalent to $\overline{K}-L\succeq0$, i.e., $L\preceq\overline{K}$.
This proves the final part. 
\end{proof}

\section{Kernel Domination and RN Representation}\label{sec:4}

This section moves from the universal model back down to the base
RKHS of the original kernel. When each composing map is subunital,
every iterate sits below the starting kernel in the positive definite
order. On the Kolmogorov space this means there is a single positive
contraction that “explains” the iterate as a compression of the original
features. Concretely, we identify each composed kernel with a bounded
operator on the base space, so comparison, bounds, and limits reduce
to ordinary operator inequalities there. The picture is minimal, no
large ambient sum is needed, and it makes the domination $K_{w}\preceq K$
and the associated bounds more transparent.\\

Fix the data from the setting in \prettyref{sec:2}. Let $K:X\times X\rightarrow B\left(H\right)$
be p.d., and let $\mathcal{H}_{\tilde{K}}$ be the canonical Kolmogorov
space for $K$ with maps 
\[
V_{x}:H\rightarrow\mathcal{H}_{\tilde{K}},\qquad V_{x}a=\tilde{K}_{\left(x,a\right)},
\]
so that 
\[
K\left(x,y\right)=V^{*}_{x}V_{y}.
\]
Recall that $J_{x}:H\rightarrow M$ is given by $J_{x}\left(a\right)=e_{\emptyset}\cdot\tilde{K}_{\left(x,a\right)}$.
That is, we identify $\mathcal{H}_{\tilde{K}}$ with the $\emptyset$-summand
of $M$ via the isometric embedding 
\[
\iota:\mathcal{H}_{\tilde{K}}\rightarrow M,\qquad\iota\left(f\right)=e_{\emptyset}\cdot f,
\]
so that $J_{x}=\iota\circ V_{x}$. 

We consider normal completely positive maps $\left\{ \Phi_{s}\right\} _{s\in S}$
with Kraus forms as in \prettyref{sec:2}, and the iterated kernels
\[
K_{w}:=\Phi_{s_{1}}\circ\cdots\circ\Phi_{s_{n}}\left(K\right),\qquad w=s_{1}\cdots s_{n}\in S^{*},
\]
with $K_{\emptyset}=K$.
\begin{lem}[Radon-Nikodym for p.d. kernels]
\label{lem:D1}Let $K,L:X\times X\rightarrow B\left(H\right)$ be
p.d. kernels with $L\preceq K$ (i.e., $K-L$ is p.d.) Let $\mathcal{H}_{\tilde{K}}$
be the canonical Kolmogorov space of $K$ with maps $V_{x}:H\rightarrow\mathcal{H}_{\tilde{K}}$
so that $K\left(x,y\right)=V^{*}_{x}V_{y}$. Then there exists a unique
operator $A\in B\left(\mathcal{H}_{\tilde{K}}\right)$ with $0\leq A\leq I_{\mathcal{H}_{\tilde{K}}}$
such that 
\[
L\left(x,y\right)=V^{*}_{x}AV_{y},\qquad x,y\in X.
\]
\end{lem}

\begin{proof}[Proof sketch]
On the algebraic span of $\left\{ V_{y}b\right\} $, define 
\[
\left\langle V_{x}a,V_{y}b\right\rangle _{L}:=\left\langle a,L\left(x,y\right)b\right\rangle _{H}.
\]
Since $L\preceq K$, Cauchy-Schwarz gives 
\[
\left|\left\langle \xi,\eta\right\rangle _{L}\right|\leq\left\Vert \xi\right\Vert _{\mathcal{H}_{\tilde{K}}}\left\Vert \eta\right\Vert _{\mathcal{H}_{\tilde{K}}},
\]
so this form is bounded by the $\mathcal{H}_{\tilde{K}}$ norm and
extends uniquely to a bounded positive operator $A$ with $0\leq A\leq I$
via $\left\langle \xi,\eta\right\rangle _{L}=\left\langle \xi,A\eta\right\rangle _{\mathcal{H}_{\tilde{K}}}$.
Uniqueness follows from the minimality of $\mathcal{H}_{\tilde{K}}$
(the span of $\left\{ V_{y}b\right\} $ is dense). 

For full details and variants, see e.g., \cite{MR2938971,MR1976867,MR247483,MR4250453}.
\end{proof}
\begin{rem*}
The operator $A$ in \prettyref{lem:D1} is called the Radon-Nikodym
derivative of $L$ with respect to $K$ (on the Kolmogorov space $\mathcal{H}_{\tilde{K}}$),
denoted by $A=dL/dK$. 
\end{rem*}
\begin{thm}[RN representation under subunitality]
\label{thm:4-2}Assume each $\Phi_{s}$ is subunital, i.e., $\Phi_{s}\left(I\right)\leq I$.
Then for every $w\in S^{*}$, there exists a unique bounded positive
operator $A_{w}\in B\left(\mathcal{H}_{\tilde{K}}\right)$ with $0\leq A_{w}\leq I_{\mathcal{H}_{\tilde{K}}}$
such that 
\[
K_{w}\left(x,y\right)=V^{*}_{x}A_{w}V_{y},\qquad x,y\in X.
\]
Equivalently, for all $x,y\in X$ and $a,b\in H$, 
\[
\left\langle a,K_{w}\left(x,y\right)b\right\rangle _{H}=\left\langle V_{x}a,A_{w}V_{y}b\right\rangle _{\mathcal{H}_{\tilde{K}}}.
\]
Moreover, with $C_{w}:=C_{s_{n}}\cdots C_{s_{1}}$ from \eqref{eq:B1-1},
one has the explicit formula
\[
A_{w}=\iota^{*}C^{*}_{w}C_{w}\iota,
\]
and hence the kernel domination $K_{w}\preceq K$ (in the usual p.d.
order) holds for all $w\in S^{*}$. 
\end{thm}

\begin{proof}
Since each $\Phi_{s}$ is subunital, part \eqref{enu:B1-2} of \prettyref{thm:B1}
gives $\left\Vert C_{s}\right\Vert \leq1$. Hence every $C_{w}$ is
a contraction and $0\leq C^{*}_{w}C_{w}\leq I_{M}$. Thus, 
\[
A_{w}:=\iota^{*}C^{*}_{w}C_{w}\iota\in B(\mathcal{H}_{\tilde{K}})
\]
is a positive contraction, $0\leq A_{w}\leq\iota^{*}I_{M}\iota=I_{\mathcal{H}_{\tilde{K}}}$. 

For $x,y\in X$ and $a,b\in H$, 
\begin{align*}
\left\langle a,K_{w}\left(x,y\right)b\right\rangle _{H} & =\left\langle C_{w}J_{x}\left(a\right),C_{w}J_{y}\left(b\right)\right\rangle _{M}\\
 & =\left\langle C_{w}\iota V_{x}a,C_{w}\iota V_{y}b\right\rangle _{M}\\
 & =\left\langle V_{x}a,\iota^{*}C^{*}_{w}C_{w}\iota V_{y}b\right\rangle _{\mathcal{H}_{\tilde{K}}}.
\end{align*}
This proves the representation $K_{w}\left(x,y\right)=V^{*}_{x}A_{w}V_{y}$
and shows that $0\leq A_{w}\leq I$. Kernel domination follows immediately:
\[
\left\langle a,\left(K-K_{w}\right)\left(x,y\right)b\right\rangle _{H}=\left\langle V_{x}a,\left(I-A_{w}\right)V_{y}b\right\rangle \geq0,
\]
so $K_{w}\preceq K$. 

Uniqueness: by minimality of the Kolmogorov decomposition, $\mathcal{H}_{\tilde{K}}$
is the closed linear span of $\left\{ V_{y}b:y\in X,b\in H\right\} $.
If $B$ is another bounded operator on $\mathcal{H}_{\tilde{K}}$
with
\[
V^{*}_{x}BV_{y}=V^{*}_{x}A_{w}V_{y}
\]
for all $x,y\in X$, then 
\[
\left\langle V_{x}a,\left(B-A_{w}\right)V_{y}b\right\rangle =0
\]
for all $x,y,a,b$, hence $B=A_{w}$. 
\end{proof}
\begin{rem}[Relation to the $M$-model]
 The operator $A_{w}$ is exactly the compression of $C^{*}_{w}C_{w}$
to the $\emptyset$-layer of $M$. Thus the RN picture on $\mathcal{H}_{\tilde{K}}$
is a compressed form of the universal model from \prettyref{thm:B1}.
In particular, 
\[
\left\Vert A_{w}\right\Vert \leq\left\Vert C_{w}\right\Vert ^{2}\leq\prod^{n}_{j=1}\left\Vert \Phi_{s_{j}}\right\Vert _{\mathrm{cb}},
\]
and if $\sup_{s}\left\Vert \Phi_{s}\right\Vert _{\mathrm{cb}}\leq L<1$
then $\left\Vert A_{w}\right\Vert \leq L^{\left|w\right|}$. 
\end{rem}

\begin{rem}
If some $\Phi_{s}$ is not subunital, the contraction property of
$C_{s}$ can fail, and so can the kernel domination $K_{w}\preceq K$.
In that case, there need not exist a global positive contraction $0\leq A_{w}\leq I$
on $\mathcal{H}_{\tilde{K}}$ with $K_{w}\left(x,y\right)=V^{*}_{x}A_{w}V_{y}$,
but the universal model still exists (\prettyref{thm:B1}). 
\end{rem}

\prettyref{lem:D1} established the RN representation for positive
definite kernels: if $L\preceq K$, then there exists a unique $0\le A\le I_{\mathcal{H}_{\tilde{K}}}$
such that $L(x,y)=V^{*}_{x}AV_{y}$. This operator $A$ acts on the
canonical Kolmogorov space of $K$ and compresses $K$ to $L$. The
same argument extends to the iterated kernels generated by the family
$\left\{ \Phi_{s}\right\} _{s\in S}$. Through the universal model
constructed in \prettyref{sec:2}, the Radon-Nikodym operator $A$
acts diagonally across the graded space $M$, producing a single compression
that intertwines all iterates. The following result makes this extension
precise.
\begin{thm}[order-monotone compression across models]
Let $K^{\left(1\right)},K^{\left(2\right)}:X\times X\rightarrow B\left(H\right)$
be p.d. kernels with $K^{\left(1\right)}\preceq K^{\left(2\right)}$.
Fix the same CP family $\left\{ \Phi_{s}\right\} _{s\in S}$ and Kraus
data as in \prettyref{sec:2}, and form the universal models 
\[
\left(M_{i},\{C^{\left(i\right)}_{s}\}_{s\in S},\{J^{\left(i\right)}_{x}\}_{x\in X}\right),\quad i=1,2,
\]
for $\left(K^{\left(i\right)},\{\Phi_{s}\}\right)$ via \prettyref{thm:B1}. 

Then there exists a unique positive contraction $A\in B\left(\mathcal{H}_{\tilde{K}_{2}}\right)$
(i.e., $A=d\tilde{K}^{\left(1\right)}/d\tilde{K}^{\left(2\right)}$)
such that, for every word $w\in S^{*}$ and all $x,y\in X$, 
\begin{equation}
K^{\left(1\right)}_{w}\left(x,y\right)=J^{\left(2\right)*}_{x}C^{\left(2\right)*}_{w}\left(I_{l^{2}\left(R\right)}\otimes A\right)C^{\left(2\right)}_{w}J^{\left(2\right)}_{y}.\label{eq:d1}
\end{equation}
\end{thm}

\begin{proof}
Consider the scalar lifts 
\[
\tilde{K}^{\left(i\right)}\left(\left(x,a\right),\left(y,b\right)\right):=\langle a,K^{\left(i\right)}\left(x,y\right)b\rangle_{H}
\]
on $X\times H$ and their RKHSs $\mathcal{H}_{\tilde{K}^{\left(i\right)}}$
with kernel sections $\tilde{K}^{\left(i\right)}_{\left(x,a\right)}$.
Since $K^{\left(1\right)}\preceq K^{\left(2\right)}$, we have $\tilde{K}^{\left(1\right)}\preceq\tilde{K}^{\left(2\right)}$
as scalar kernels. By \prettyref{lem:D1}, there exists a unique RN
derivative $A=d\tilde{K}^{\left(1\right)}/d\tilde{K}^{\left(2\right)}$.
In particular, $A\in B\left(\mathcal{H}_{\tilde{K}_{2}}\right)$ with
$0\leq A\leq I$. 

Both models $M_{i}$ are constructed as 
\[
M_{i}=l^{2}\left(R\right)\otimes\mathcal{H}_{\tilde{K}^{\left(i\right)}}\simeq\bigoplus_{\rho\in R}\mathcal{H}_{\tilde{K}^{\left(i\right)}},
\]
with the same grading set $R$ (built only from $S$ and the Kraus
index sets). Write elements of $M_{i}$ as $\sum_{\rho}e_{\rho}\cdot F^{\left(i\right)}_{\rho}$,
where $F^{\left(i\right)}_{\rho}\in\mathcal{H}_{\tilde{K}^{\left(i\right)}}$
and $\left\{ e_{\rho}\right\} $ is the canonical ONB of $l^{2}\left(R\right)$. 

By \prettyref{thm:B1}, for each $i=1,2$, 
\[
K^{\left(i\right)}_{w}\left(x,y\right)=J^{\left(i\right)*}_{x}C^{\left(i\right)*}_{w}C^{\left(i\right)}_{w}J^{\left(i\right)}_{y}.
\]
Fix $w=s_{1}\cdots s_{n}\in S^{*}$. Expanding by the definition of
$C^{\left(2\right)}_{w}$ and the orthogonality of the $R$-summands,
\[
C^{\left(2\right)}_{w}J^{\left(2\right)}_{x}a=\sum_{\left(r_{1},\dots,r_{n}\right)}e_{\left(r_{1},\dots,r_{n}\right)}\cdot\tilde{K}^{\left(2\right)}_{\left(x,A_{s_{n},r_{n}}\cdots A_{s_{1},r_{1}}a\right).}
\]
Therefore, 
\begin{eqnarray*}
 &  & \left\langle C^{\left(2\right)}_{w}J^{\left(2\right)}_{x}a,\left(I\otimes A\right)C^{\left(2\right)}_{w}J^{\left(2\right)}_{y}b\right\rangle _{M_{2}}\\
 & = & \sum_{\left(r_{1},\dots,r_{n}\right)}\left\langle \tilde{K}^{\left(2\right)}_{\left(x,A_{s_{n},r_{n}}\cdots A_{s_{1},r_{1}}a\right)},A\tilde{K}^{\left(2\right)}_{\left(y,A_{s_{n},r_{n}}\cdots A_{s_{1},r_{1}}b\right)}\right\rangle _{\mathcal{H}_{\tilde{K}^{\left(2\right)}}}\\
 & = & \sum_{\left(r_{1},\dots,r_{n}\right)}\left\langle \tilde{K}^{\left(1\right)}_{\left(x,A_{s_{n},r_{n}}\cdots A_{s_{1},r_{1}}a\right)},\tilde{K}^{\left(1\right)}_{\left(y,A_{s_{n},r_{n}}\cdots A_{s_{1},r_{1}}b\right)}\right\rangle _{\mathcal{H}_{\tilde{K}^{\left(1\right)}}}\\
 & = & \left\langle a,\left(\Phi_{s_{1}}\circ\cdots\circ\Phi_{s_{n}}\right)\left(K^{\left(1\right)}\left(x,y\right)\right)b\right\rangle _{H}\\
 & = & \left\langle a,K^{\left(1\right)}_{w}\left(x,y\right)b\right\rangle _{H}
\end{eqnarray*}
and so \eqref{eq:d1} follows. Uniqueness is from that of RN derivative
on $\mathcal{H}_{\tilde{K}^{\left(2\right)}}$. 
\end{proof}

\section{Random Iterations}

This section lets the composing map vary randomly at each step and
asks for the typical exponential growth (or decay) rate of the pushed-forward
features. On the universal model from \prettyref{sec:2}, a random
word becomes a random product of the creation operators, subadditivity
then forces an almost-sure limit for the average log-norm (the Lyapunov
exponent). Translating back to kernels, this single number controls
how fast the composed kernels amplify or contract on any fixed pair
of points, and it yields clean almost-sure bounds for their entries
and operator norms. In the deterministic unital case the exponent
collapses to zero, matching the non-decaying limit kernel identified
in \prettyref{sec:3}.\\

Fix the data from the setting in \prettyref{sec:2}. In particular,
for each $z\in S$ we have the bounded operator $C_{z}:M\rightarrow M$
and, for a word $w=z_{1}\cdots z_{n}\in S^{*}$, the product $C_{w}:=C_{z_{n}}\cdots C_{z_{1}}$
and the kernel realization 
\[
K_{w}\left(x,y\right)=J^{*}_{x}C^{*}_{w}C_{w}J_{y}.
\]

Let $\mu$ be a probability measure on $S$. Work on the canonical
path space $\left(\Omega,\mathcal{F},\mathbb{P}\right)$, where 
\[
\Omega:=S^{\mathbb{N}},\quad\mathbb{P}:=\mu^{\otimes\mathbb{N}}
\]
with the product $\sigma$-algebra $\mathcal{F}$. Let $z_{n}:\Omega\rightarrow S$
be the coordinate maps, $z_{n}\left(\omega\right)=\omega_{n}$ (i.i.d.
with law $\mu$) For $\omega\in\Omega$, set the random word of length
$n$ 
\[
\xi_{n}\left(\omega\right):=z_{1}\left(\omega\right)\cdots z_{n}\left(\omega\right),
\]
and 
\[
C_{\xi_{n}\left(\omega\right)}=C_{z_{n}\left(\omega\right)}\cdots C_{z_{1}\left(\omega\right)}.
\]
We assume the integrability condition
\begin{equation}
\mathbb{E}\left[\log^{+}\left\Vert C_{z_{1}}\right\Vert \right]<\infty\label{eq:5-1}
\end{equation}
where $\log^{+}t=\max\left\{ 0,\log t\right\} $. 

Define the left shift $\theta:\Omega\rightarrow\Omega$, 
\[
\theta\left(\omega_{1},\omega_{2},\omega_{3},\dots\right)=\left(\omega_{2},\omega_{3},\dots\right).
\]
Then $\theta$ is measure-preserving and $z_{n+1}=z_{n}\circ\theta$. 
\begin{thm}
\label{thm:5-1}Under \eqref{eq:5-1}, there exists a deterministic
constant $\lambda\in[-\infty,\infty)$ such that 
\[
\lim_{n\rightarrow\infty}\frac{1}{n}\log\left\Vert C_{\xi_{n}\left(\omega\right)}\right\Vert =\lambda\qquad\mathbb{P}\text{-a.e. }\omega.
\]
Moreover, 
\[
\lambda=\inf_{n\geq1}\frac{1}{n}\mathbb{E}\left[\log\left\Vert C_{\xi_{n}}\right\Vert \right]=\lim_{n\rightarrow\infty}\frac{1}{n}\mathbb{E}\left[\log\left\Vert C_{\xi_{n}}\right\Vert \right].
\]
\end{thm}

\begin{proof}
For $n\geq1$, set 
\[
X_{n}\left(\omega\right):=\log\left\Vert C_{\xi_{n}\left(\omega\right)}\right\Vert \in[-\infty,\infty).
\]
By submultiplicativity of the operator norm and the definition of
$\xi_{n}$, 
\begin{align*}
X_{m+n}\left(\omega\right) & =\log\left\Vert C_{\xi_{m+n}\left(\omega\right)}\right\Vert \\
 & =\log\left\Vert C_{z_{m+n}\left(\omega\right)}\cdots C_{z_{1}\left(\omega\right)}\right\Vert \\
 & \leq\log\left\Vert C_{z_{m+n}\left(\omega\right)}\cdots C_{z_{m+1}\left(\omega\right)}\right\Vert +\log\left\Vert C_{z_{m}\left(\omega\right)}\cdots C_{z_{1}\left(\omega\right)}\right\Vert \\
 & =X_{n}\left(\theta^{m}\omega\right)+X_{m}\left(\omega\right).
\end{align*}
Thus $\left(X_{n}\right)_{n\geq1}$ is a subadditive, stationary process
over the measure-preserving transformation $\theta$. The integrability
hypothesis \eqref{eq:5-1} yields $\mathbb{E}\left[X^{+}_{1}\right]<\infty$.
Here, stationary with respect to $\theta$ means 
\[
X_{n}\circ\theta^{m}\overset{d}{=}X_{n}
\]
for all $m\geq0$. Equivalently, 
\[
\mathbb{E}\left[g\left(X_{n}\circ\theta^{m}\right)\right]=\mathbb{E}\left[g\left(X_{n}\right)\right]
\]
for all bounded measurable $g$.

By Kingman's subadditive ergodic theorem (see e.g., \cite{MR254907,MR356192,MR797411,MR3930614}),
there exists a deterministic $\lambda\in[-\infty,\infty)$ such that
\[
\frac{1}{n}X_{n}\left(\omega\right)\rightarrow\lambda\;\text{for \ensuremath{\mathbb{P}}-a.e. }\omega,
\]
and moreover, 
\[
\lambda=\inf_{n\geq1}\frac{1}{n}\mathbb{E}\left[X_{n}\right]=\lim_{n\rightarrow\infty}\frac{1}{n}\mathbb{E}\left[X_{n}\right].
\]
\end{proof}
\begin{cor}[kernel scalar growth]
\label{cor:5-2} Let $\lambda$ be as in \prettyref{thm:5-1}. For
all $x,y\in X$ and $a,b\in H$, 
\[
\limsup_{n\rightarrow\infty}\frac{1}{n}\log\left|\left\langle a,K_{\xi_{n}\left(\omega\right)}\left(x,y\right)b\right\rangle _{H}\right|\leq2\lambda,\quad\mathbb{P}\text{-a.e. }\omega.
\]
 
\end{cor}

\begin{proof}
By the realization $K_{w}\left(x,y\right)=J^{*}_{x}C^{*}_{w}C_{w}J_{y}$
and Cauchy-Schwarz, 
\begin{align*}
\left|\left\langle a,K_{\xi_{n}\left(\omega\right)}\left(x,y\right)b\right\rangle _{H}\right| & =\left|\left\langle C_{\xi_{n}}J_{x}\left(a\right),C_{\xi_{n}}J_{y}\left(b\right)\right\rangle _{M}\right|\\
 & \leq\left\Vert C_{\xi_{n}}\right\Vert ^{2}_{M}\left\Vert J_{x}\left(a\right)\right\Vert _{M}\left\Vert J_{y}\left(b\right)\right\Vert _{M}.
\end{align*}
Take logs, divide by $n$, and use \prettyref{thm:5-1}.
\end{proof}
\begin{cor}[operator-norm bound]
 For any random word $\xi_{n}$, 
\[
\limsup_{n\rightarrow\infty}\frac{1}{n}\log\left\Vert K_{\xi_{n}\left(\omega\right)}\left(x,y\right)\right\Vert \leq2\lambda,\quad\mathbb{P}\text{-a.e. }\omega.
\]
\end{cor}

\begin{proof}
Recall that, for any word $w=s_{1}\cdots s_{n}\in S^{*}$ and $x,y\in X$,
\begin{align*}
\left\Vert K_{w}\left(x,y\right)\right\Vert _{B\left(H\right)} & \leq\left\Vert C_{w}\right\Vert ^{2}_{M}\left\Vert J_{x}\right\Vert _{M}\left\Vert J_{y}\right\Vert _{M}\\
 & =\left\Vert C_{w}\right\Vert ^{2}_{M}\left\Vert K\left(x,x\right)\right\Vert ^{1/2}_{B\left(H\right)}\left\Vert K\left(y,y\right)\right\Vert ^{1/2}_{B\left(H\right)}.
\end{align*}
Use the random word $\xi_{n}$, and apply \prettyref{thm:5-1} to
$\left\Vert C_{\xi_{n}}\right\Vert _{M}$ to get the a.s. $\limsup$. 
\end{proof}
\begin{cor}[RN formulation under subunitality]
 Assume each $\Phi_{z}$ is subunital. Then for any $w\in S^{*}$
there exists a positive contraction $A_{w}\in B\left(\mathcal{H}_{\tilde{K}}\right)$
with 
\[
K_{w}\left(x,y\right)=V^{*}_{x}A_{w}V_{y}
\]
with $0\leq A_{w}\leq I$ (see \prettyref{thm:4-2}). Hence, for any
$x,y\in X$, 
\[
\limsup_{n\rightarrow\infty}\frac{1}{n}\log\left\Vert K_{\xi_{n}\left(\omega\right)}\left(x,y\right)\right\Vert \leq\limsup_{n\rightarrow\infty}\frac{1}{n}\log\left\Vert A_{\xi_{n}\left(\omega\right)}\right\Vert \leq2\lambda\quad\text{a.s.}
\]
In particular, if $\lambda<0$, then $\left\Vert K_{\xi_{n}\left(\omega\right)}\left(x,y\right)\right\Vert \rightarrow0$
exponentially fast almost surely.
\end{cor}

\begin{proof}
This follows from \prettyref{sec:4}, where $A_{w}=\iota^{*}C^{*}_{w}C_{w}\iota$,
so $0\leq A_{w}\leq I$ and $\left\Vert A_{w}\right\Vert \leq\left\Vert C_{w}\right\Vert ^{2}$,
combined with \prettyref{thm:5-1}. 
\end{proof}
\begin{cor}[uniform cb-bound]
\label{cor:5-5} If $\sup_{z\in S}\left\Vert \Phi_{z}\right\Vert _{\mathrm{cb}}\leq L<\infty$,
then $\left\Vert C_{z}\right\Vert \leq L^{1/2}$ for all $z$, and
for any word $w$ of length $\left|w\right|$, 
\[
\left\Vert K_{w}\left(x,y\right)\right\Vert \leq L^{\left|w\right|}\left\Vert K\left(x,x\right)\right\Vert ^{1/2}\left\Vert K\left(y,y\right)\right\Vert ^{1/2}.
\]
In particular, if $L<1$, then for the random words $\xi_{n}$, 
\[
\left\Vert K_{\xi_{n}\left(\omega\right)}\left(x,y\right)\right\Vert \leq L^{n}\left\Vert K\left(x,x\right)\right\Vert ^{1/2}\left\Vert K\left(y,y\right)\right\Vert ^{1/2},\qquad\forall n,
\]
and thus $\lambda\leq\frac{1}{2}\log L<0$ and exponential decay holds
almost surely. 
\end{cor}

\begin{proof}
The bound $\left\Vert C_{z}\right\Vert \leq\left\Vert \Phi_{z}\right\Vert ^{1/2}_{\mathrm{cb}}$
is \prettyref{thm:B1}, part \eqref{enu:B1-2}. The rest is immediate. 
\end{proof}
\begin{rem}[UCP, spectral radius]
 Let $\Phi_{z}\equiv\Phi$ be unital CP, so $C_{z}=C$ and $C_{\xi_{n}}=C^{n}$.
Then 
\[
\lambda=\lim_{n\rightarrow\infty}\frac{1}{n}\log\left\Vert C^{n}\right\Vert =\log r\left(C\right)\leq0,
\]
where $r\left(C\right)$ is the spectral radius. 

By \prettyref{thm:3-1}, $C^{*n}C^{n}\xrightarrow{s}P$, where $P$
is the orthogonal projection onto $M_{iso}=\left\{ v:\left\Vert C^{n}v\right\Vert =\left\Vert v\right\Vert ,\forall n\right\} $.
\begin{enumerate}
\item If $P\neq0$, then $r\left(C\right)=1$ and $\lambda=0$. In this
case, 
\[
K_{n}\left(x,y\right)=J^{*}_{x}C^{*n}C^{n}J_{y}\rightarrow J^{*}_{x}PJ_{y}=\overline{K}\left(x,y\right).
\]
\item If $P=0$, then $K_{n}\left(x,y\right)\rightarrow0$ for each $x,y\in X$.
Hence $\lambda=\log r\left(C\right)\leq0$, with $\lambda\leq0$ giving
exponential decay and $\lambda=0$ allowing subexponential decay. 
\end{enumerate}
\end{rem}

\bibliographystyle{amsalpha}
\bibliography{ref}

\end{document}